\newif\ifarxiv
\newif\ifarticleclass
\PassOptionsToPackage{nameinlink,capitalize}{cleveref}
\PassOptionsToPackage{dvipsnames}{xcolor}
\documentclass[cleveref]{l4dc2024}\arxivtrue

\input{TeX/Preamble.tex}

\newcommand{\TheTitle}{On the convergence of adaptive first order methods:\texorpdfstring{\\}{ }proximal gradient and alternating minimization algorithms}
\newcommand{\TheShortTitle}{On the convergence of adaptive first order methods: PG and AMA}

\newcommand{\TheFunding}{%
	Work supported by:
	the Research Foundation Flanders (FWO) postdoctoral grant 12Y7622N and research projects G081222N, G033822N, and G0A0920N;
	European Union's Horizon 2020 research and innovation programme under the Marie Sk\l odowska-Curie grant agreement No. 953348;
	Japan Society for the Promotion of Science (JSPS) KAKENHI grants JP21K17710 and JP24K20737.%
}
\newcommand{\TheKeywords}{%
	Convex minimization%
\AND
	proximal gradient method%
\AND
	alternating minimization algorithm%
\AND
	locally Lipschitz gradient%
\AND
	linesearch-free adaptive stepsizes%
}
\newcommand{\TheSubjclass}{%
	\amsmscLink{65K05}%
\AND
	\amsmscLink{90C06}%
\AND
	\amsmscLink{90C25}%
\AND
	\amsmscLink{90C30}%
\AND
	\amsmscLink{90C47}%
}

\ifarticleclass
	\title{\TheTitle\thanks{\TheFunding}}
	\author{Puya Latafat\thanks{\TheAddressKU}\and Andreas Themelis\thanks{\TheAddressKUJ}\and Panagiotis Patrinos\footnotemark[2]}
	\date{}
\else
	\title[\TheShortTitle]{\TheTitle}
	\author{%
		\Name{Puya Latafat} \Email{puya.latafat@kuleuven.be}\\
		\addr{\TheAddressKU}
	\AND
		\Name{Andreas Themelis} \Email{andreas.themelis@ees.kyushu-u.ac.jp}\\
		\addr{\TheAddressKUJ}
	\AND
		\Name{Panagiotis Patrinos} \Email{panos.patrinos@esat.kuleuven.be}\\
		\addr{\TheAddressKU}
	}
\fi

\showgrayfalse

\begin{document}
	\maketitle

	\begin{abstract}%
Building upon recent works on linesearch-free adaptive proximal gradient methods, this paper proposes \refadaPG, a framework that unifies and extends existing results by providing larger stepsize policies and improved lower bounds.
Different choices of the parameters \(\rp\) and \(r\) are discussed and the efficacy of the resulting methods is demonstrated through numerical simulations.
In an attempt to better understand the underlying theory, its convergence is established in a more general setting that allows for time-varying parameters.
Finally, an adaptive alternating minimization algorithm is presented by exploring the dual setting. This algorithm not only incorporates additional adaptivity, but also expands its applicability beyond standard strongly convex settings.

	\end{abstract}


	\ifarticleclass

		\pdfbookmark[1]{\contentsname}{Contents}
		\tableofcontents
	\fi

	\section{Introduction}
The proximal gradient ({\algnamefont PG}) method is the natural extension of gradient descent for constrained and nonsmooth problems.
It addresses nonsmooth minimization problems by splitting them as
\[\tag{P}\label{eq:PG:P}
  \minimize_{x\in\R^n}\; \varphi(x) \coloneqq f(x)+g(x),
\]
where \(f\) is here assumed \emph{locally} Lipschitz differentiable, and \(g\) possibly nonsmooth but with an easy-to-compute proximal mapping, while both being convex (see \cref{ass:PG:P} for details).
It has long been known that performance of first-order methods can be drastically improved by an appropriate stepsize selection as evident in the success of linesearch based approaches.

Substantial effort has been devoted to developing adaptive methods.
Most notably, in the context of stochastic (sub)gradient descent, numerous adaptive methods have been proposed starting with \cite{duchi2011adaptive}.
We only point the reader to few recent works in this area \cite{li2019convergence,ward2019adagrad,yurtsever2021three,ene2021adaptive,defazio2022grad,ivgi2023dog}.
However, although applicable to a more general setting, such approaches tend to suffer from diminishing stepsizes, which can hinder their performance.

Closer to our setting are recent works \cite{grimmer2023accelerated,altschuler2023acceleration} which consider smooth optimization problems, and propose predefined stepsize patterns.
These methods obtain accelerated worst-case rates under global Lipschitz continuity assumptions.
We also mention the recent work \cite{li2023simple} in the constrained smooth setting which, while also being bound to a global Lipschitz continuity assumption, uses an adaptive estimate for the Lipschitz modulus and achieves an \emph{accelerated} worst-case rate.

In this paper, we extend recent results pioneered in \cite{malitsky2020adaptive} and later further developed in \cite{latafat2023adaptive,malitsky2023adaptive}, where novel \mbox{(self-)}\allowbreak adaptive schemes are developed.
We provide a unified analysis that bridges together and improves upon all these works by enabling larger stepsizes and, in some cases, providing tighter lower bounds.
Adaptivity refers to the fact that, in contrast to linesearch methods that employ a \emph{look forward} approach based on trial and error to ensure a sufficient descent in the cost, we \emph{look backward} to yield stepsizes only based on past information.
Specifically, we estimate the Lipschitz modulus of \(\nabla f\) at consecutive iterates \(x^{k-1},x^k\in\R^n\) generated by the algorithm using the quantities
\begin{equation}\label{eq:lL}
		\lk
	\coloneqq
		\frac{
			\innprod{\nabla f(x^k)-\nabla f(x^{k-1})}{x^k-x^{k-1}}
		}{
			\|x^k-x^{k-1}\|^2
		}
\quad\text{and}\quad
		L_k
	\coloneqq
		\frac{
			\|\nabla f(x^k)-\nabla f(x^{k-1})\|
		}{
			\|x^k-x^{k-1}\|
		}.
\end{equation}
Throughout, we stick to the convention \(\frac00=0\) so that \(\lk\) and \(L_k\) are well-defined, positive real numbers.
In addition, we adhere to \(\frac10 = \infty\).
Note also that
\begin{equation}\label{eq:lk<=ck}
	\lk
\leq
	L_k
\leq
	L_{f,\V}
\end{equation}
holds whenever \(L_{f,\V}\) is a Lipschitz modulus for \(\nabla f\) on a convex set \(\V\) containing \(x^{k-1}\) and \(x^k\). 
Despite the mere dependence of these quantities on the previous iterates, they provide a sufficiently refined estimate of the local geometry of \(f\).
In fact, a carefully designed stepsize update rule not only ensures that the stepsize sequence is separated from zero, but also that a sufficient descent-type inequality can be indirectly ensured between \((x^{k+1}, x^k)\) and \((x^k, x^{k-1})\) without any backtracks.

The ultimate deliverable of this manuscript is the general adaptive framework outlined in \cref{alg:PG:pinuxik}.
A special case of it is here condensed into a two-parameter simplified algorithm.

\begin{myalg}[alg:adaPG]{{\algnamefont AdaPG}\(^{\rp,r}\)}%
\begin{tabular}[t]{@{}l@{}}
	Fix \(x^{-1}\in\R^n\) and \(\gamma_0=\gamma_{-1}>0\).
	With \(\lk\) and \(L_k\) as in \eqref{eq:lL}, starting from
\\
	\(x^0=\prox_{\gamma_0g}(x^{-1}-\gamma_0\nabla f(x^{-1}))\),
	iterate for \(k = 0, 1, \ldots\)
\end{tabular}
\vspace*{-.75\baselineskip}%
\begin{subequations}
	\begin{align}\label{eq:adaPG}
		\gamk*
	={} &
		\textstyle
		\gamk\min\set{
			\sqrt{%
				\frac{1}{\rp}+\frac{\gamk}{\gamma_{k-1}}
				\vphantom{
					\frac{
						1 - \frac{r}{\rp}
					}{
						\left[
							(1-2r)
							+
							\gamk^2 L_k^2
							+
							2\gamk\lk(r-1)
						\right]_+
					}
				}
			},
		~
			\sqrt{
				\frac{
					1 - \frac{r}{\rp}
				}{
					\left[
						\gamk^2 L_k^2
						+
						2\gamk\lk(r-1)
						-
						(2r-1)
					\right]_+
				}
			}
		}
	\\
		x^{k+1}
	={} &
		\prox_{\gamk*g}(x^k-\gamk*\nabla f(x^k))
	\end{align}
\end{subequations}
\end{myalg}

\begin{theorem}\label{thm:PG:simplified}%
	Under \cref{ass:PG:P}, for any \(\rp>r\geq\frac12\) the sequence \(\seq{x^k}\) generated by \refadaPG\ converges to some \(x^\star\in\argmin\varphi\).
	If in addition \(\rp \leq \tfrac12(3+\sqrt5)\), then
	\[
		\gamk
	\geq
		\gamma_{\rm min}
	\coloneqq
		\sqrt{\tfrac{1-\frac{r}{\rp}}{\max\set{1,\rp}}}
		\tfrac{1}{L_{f,\V}}
	\quad\text{holds for all }
		k
	\geq
		2\bigl\lceil\log_{1+\frac{1}{\rp}}\bigl(\tfrac{1}{\gamma_0L_{f,\V}}\bigr)\bigr\rceil_+,
	\]
	where \(L_{f,\V}\) is a Lipschitz modulus for \(\nabla f\) on a convex and compact set \(\V\) that contains \(\seq{x^k}\).
	Moreover,
	\(
		\min_{k\leq K} (\varphi(x^k) - \min \varphi)
	\leq
		\frac{\U_1(x^\star)}{\sum_{k=1}^{K+1}\gamk}
	\)
	holds for every \(K\geq1\), where \(\U_1(x^\star)\) is as in \eqref{eq:PG:Uk}.
\end{theorem}

The above worst-case sublinear rate depends on the aggregate of the stepsize sequence, providing a partial explanation for the fast convergence of the algorithm observed in practice.
\refadaPG\ and \cref{thm:PG:simplified} are particular instances of the general framework provided in \cref{sec:PG}, see \cref{rem:relation,thm:PG:pinuxik} for the details.
Specific choices of the parameters \(\rp,r\) nevertheless allow \refadaPG\ to embrace and extend existing algorithms:
\begin{itemize}[leftmargin=*]
\item
	\(r=\frac12\) and \(\rp=1\).
	Then,
	\(
		\gamk*
	=
		\gamk\min\set{
			\sqrt{%
				1+\frac{\gamk}{\gamma_{k-1}}
			},
			~
			\frac{1}{
				\sqrt{
					2\left[\gamk^2L_k^2-\gamk\lk\right]_+
				}
			}
		}
	\)
	coincides with the update in \cite[Alg. 2.1]{latafat2023adaptive} with second term improved by a \(\sqrt2\) factor.

\item
	Owing to the relation \(\gamk^2L_k^2-\gamk\lk\leq\gamk^2L_k^2\), the case above is also a proximal extension of \cite[Alg. 1]{malitsky2023adaptive} which considers
	\(
		\gamk*
	=
		\min\set{
			\gamk
			\sqrt{%
				1+\frac{\gamk}{\gamma_{k-1}}
			},
			~
			\frac{1}{\sqrt{2}L_k}
		}
	\)
	when \(g=0\), and which in turn is also a strict improvement over the previous work \cite{malitsky2020adaptive}.

\item
	\(r=\frac34\) and \(\rp=\frac32\).
	Then,
	\(
		\gamk*
	=
		\gamk\min\set{
			\sqrt{
				\frac{2}{3}+\frac{\gamk}{\gamma_{k-1}}
			},
		~
			\frac{
				1
			}{
				\sqrt{
					[2\gamk^2 L_k^2-1-\gamk\lk]_+
				}
			}
		}
	\)
	recovers the update rule \cite[Alg. 2]{malitsky2023adaptive} (in fact tighter because of the extra \(-\gamk\lk\) term).
\end{itemize}
The interplay between the parameters can then be understood by noting that \(\sqrt{\frac{1}{\rp}+\frac{\gamk}{\gamma_{k-1}}}\) allows the algorithm to recover from a potentially small stepsize, which can only decrease for a controlled number of iterations and will then rapidly enter a phase where it increases linearly until a certain threshold is reached, see the proof of \cref{thm:PG:pinuxik}.
A smaller \(\rp\) allows for a more aggressive recovery, but comes at the cost of more conservative second term.
As for \(r\), values in the range \([\nicefrac12,1]\), such as in the combinations reported in \cref{tab:nu1}, work well in practice.

\vspace{.5ex}%
\par\noindent
	\begin{minipage}{0.44\linewidth}
		\definecolor{lred}{RGB}{255,204,204}  
\definecolor{lblue}{RGB}{204,229,255}  
\definecolor{lgreen}{RGB}{204,255,204}  
\definecolor{peach}{RGB}{255,229,204}  
\definecolor{lorange}{RGB}{255,204,153}  
\centering
\begin{tabular}{@{}ccc>{}c@{}}
	\hline
	\(1 - \nicefrac{r}{\rp}\) & \(\rp\) & \(r\) & \(\gamma_{\rm min}L_{f, \V}\) \\
	\hline
	\hline
	\rowcolor{green!05}
	\(\nicefrac14\) & \(\nicefrac{10}{9}\) & \(\nicefrac{5}{6}\) & \(\nicefrac32\sqrt{\nicefrac1{10}} \approx 0.47\)  \\
	\hline
	\(\nicefrac25\) & \(\nicefrac{8}{5}\) & \(\nicefrac{24}{25}\) & \(\nicefrac12\)\\
	\hline
	\rowcolor{green!05}
	\(\nicefrac12\) & \(\nicefrac53\) & \(\nicefrac{5}{6}\) & \(\sqrt{\nicefrac3{10}} \approx 0.55\)\\
	\hline
	\rowcolor{green!05}
	\(\nicefrac12\) & \(\nicefrac32\) & \(\nicefrac{3}{4}\) & \(\nicefrac1{\sqrt3}\approx 0.57\)\\
	\hline
	\rowcolor{orange!05}
	\(\nicefrac12\) & 1 & \(\nicefrac{1}{2}\) & \(\nicefrac1{\sqrt2} \approx 0.71\)\\
	\hline
	\(\nicefrac35\) & \(\nicefrac52\) & \(1\) & \(\nicefrac{\sqrt{6}}5 \approx 0.49\)\\
	\hline
\end{tabular}
	\end{minipage}
	\hfill
	\begin{minipage}{0.55\linewidth}
		\refstepcounter{table}%
		\label[table]{tab:nu1}%
		\textbf{Table \thetable.}
		Suggested options for \(\rp\) and \(r\) in \protect\refadaPG{}.
		Green cells strike a nice balance between aggressive increases and large lower bounds (\(\gamma_{\rm min}\)) for the stepsize sequence, while the orange cell yields the largest theoretical lower bound.
		Here, \(L_{f,\V}\) is a local Lipschitz modulus for \(\nabla f\) as in \cref{thm:PG:simplified}.
	\end{minipage}
\vspace{.5ex}%

As a final contribution, an adaptive variant of the \emph{alternating minimization algorithm} ({\algnamefont AMA}) of \cite{tseng1991applications} is proposed that addresses composite problems of the form
\[\tag{CP}\label{eq:CP}
  \minimize_{x\in\R^n}\psi_1(x)+\psi_2(Ax).
\]
{\algnamefont AMA} is particularly interesting in settings where \(\psi_1\) is either nonsmooth or its gradient is computationally demanding.
Its convergence was established in \cite{tseng1991applications} by framing it as the dual form of the splitting method introduced in \cite{gabay1983application}, and acceleration techniques have also been adapted to this setting \cite{goldstein2014fast}.
In contrast to existing methods, ours not only incorporates an adaptive stepsize mechanism but also relaxes the strong convexity assumption to mere \emph{local} strong convexity, see \cref{ass:AMA} for details.
\ifarxiv
	\par~
\else
	Due to space limitations, some proofs are deferred to the preprint version \cite{latafat2023convergence}.
\fi

	\section{A general framework for adaptive proximal gradient methods}\label{sec:PG}%
In this section we consider plain proximal gradient iterations of the form
\begin{equation}\label{eq:PG}
	x^{k+1}=\prox_{\gamk*g}\bigl(x^k-\gamk*\nabla f(x^k)\bigr),
\end{equation}
where \(\seq{\gamk}\) is a sequence of strictly positive stepsize parameters.
The main oracles of the method are gradient and proximal maps (see \cite[\S6]{beck2017first} for examples of \emph{proximable} functions).
Whenever \(g\) is convex, for any \(\gamma>0\) it is well known that \(\prox_{\gamma g}\) is \emph{firmly nonexpansive} \cite[\S4.1 and Prop. 12.28]{bauschke2017convex}, a property stronger than Lipschitz continuity.
We here show that even when the stepsizes are time-varying as in \eqref{eq:PG} a similar property still holds for the iterates therein.
This fact is a refinement of \cite[Lem. 12]{malitsky2023adaptive} that follows after an application of Cauchy-Schwarz and that will be used in our main descent inequality.

\begin{lemma}[FNE-like inequality]\label{thm:MMtrick}%
	Suppose that \(g\) is convex and that \(f\) is differentiable.
	Then, for any \(\seq{\gamk}\subset\R_{++}\) and with \(\Hk\coloneqq\id-\gamk\nabla f\), proximal gradient iterates \eqref{eq:PG} satisfy
	\begin{equation}
		\|x^{k+1}-x^k\|^2
	\leq
		\rhok*\innprod{\Hk(x^{k-1})-\Hk(x^k)}{x^k-x^{k+1}}
	\leq
		\rhok*^2\|\Hk(x^{k-1})-\Hk(x^k)\|^2.
	\end{equation}

\end{lemma}

Throughout, we study problem \eqref{eq:PG:P} under the following assumptions.

\begin{assumption}[Requirements for problem \eqref{eq:PG:P}]\label{ass:PG:P}~
	\begin{enumeratass}
	\item
		\(\func{f}{\R^n}{\R}\) is convex and has locally Lipschitz continuous gradient.
	\item
		\(\func{g}{\R^n}{\Rinf}\) is proper, lsc, and convex.
	\item
		There exists \(x^\star \in \argmin f+g\).
	\end{enumeratass}
\end{assumption}

The main adaptive framework, involving two time-varying parameters \(\pik,\xik\), is given in \cref{alg:PG:pinuxik}.
The shorthand notation \(\one_{\gamk\lk\geq1}\) equals 1 if \(\gamk\lk\geq1\) and 0 otherwise, while for \(t\in\R\) we denote \([t]_+\coloneqq\max\set{0,t}\).

\begin{algorithm}[htb]
	\caption{General adaptive proximal gradient framework}%
	\label{alg:PG:pinuxik}%
\begin{algorithmic}[1]
\itemsep=3pt
\Require
		\begin{tabular}[t]{@{}l@{}}
			starting point \(x^{-1}\in\R^n\),~
			stepsizes \(\gamma_0=\gamma_{-1}>0\)
		\\
			parameters
			\(\frac12<\pimin\leq\pimax\),~
			\(0<\ximin\leq2\pimin-1\),~
		\end{tabular}
\Initialize
	\(
		x^0
	=
		\prox_{\gamma_0g}(x^{-1}-\gamma_0\nabla f(x^{-1}))
	\),~
	\(\rho_0=1\),~
	\(\pi_0\in[\pimin,\pimax]\),~
	\(\xi_0\geq\ximin\)

\item[\algfont{Repeat for} \(k=0,1,\ldots\) until convergence]

\State\label{state:pinuxik}%
	Let \(\lk\) and \(L_k\) be as in \eqref{eq:lL}, and choose \(\pik*,\xik*\) such that
	\[
		\textstyle
		\xik*\geq\ximin,
	\quad
		\rk*\coloneqq\frac{\pik*}{1+\xik*}\geq\frac12,
	\quad
		\pimin
	\leq
		\pik*
	\leq
		\min\set{\pimax,\pik+\one_{\gamk\lk\geq1}}
	\]

\State \label{state:pinuxik:gamk*}%
	\(
		\gamk*
	=
		\gamk\min\set{
			\sqrt{
				\frac{1+\pik\rhok}{\pik*}
			},
			~
			\sqrt{
				\frac{\rk*}{\pik*}
				\frac{
					\xik
				}{
					\left[
						\gamk^2L_k^2
						+
						2\gamk\lk\left(\rk*-1\right)
						-
						(2\rk*-1)
					\right]_+
				}
			}
		}
	\)

\State
	Set \(\rhok*=\frac{\gamk*}{\gamk}\) and update
	\(
		x^{k+1}
	=
		\prox_{\gamk*g}(x^k-\gamk*\nabla f(x^k))
	\)
\end{algorithmic}

\end{algorithm}

\ifarticleclass
	\begin{remark}[relation to \refadaPG]%
\else
	\begin{remark}\textbf{(relation to \refadaPG).}
\fi
	\label{rem:relation}%
	Whenever \(\pik\equiv\pimin=\pimax\eqqcolon\rp\) and \(\xik\equiv\ximin\eqqcolon\xi\), the conditions in \cref{alg:PG:pinuxik} reduce to
	\(
		\rp > \tfrac12
	\),
	\(
		r = \tfrac{\rp}{\xi + 1} \geq \tfrac12
	\),
	and
	\(
		\xi = \tfrac{\rp}{r} -1 > 0
	\);
	equivalently, \(\rp > r \geq \frac12\) as in \cref{thm:PG:simplified}.
\end{remark}


In what follows, for \(x\in\dom\varphi\) we adopt the notation
\begin{equation}
	P_k(x)\coloneqq\varphi(x^k)-\varphi(x).
\end{equation}
Our convergence analysis revolves around showing that under appropriate stepsize update and parameter selection the function
\begin{equation}\label{eq:PG:Uk}
		\U_k(x)
	\coloneqq{}
		\tfrac{1}{2}\|x^k-x\|^2
		+
		\gamk(1+\pik\rhok)P_{k-1}(x)
		+
		\tfrac{\xik}{2}\|x^k-x^{k-1}\|^2,
\end{equation}
monotonically decreases along the iterates for all \(x\in\argmin\varphi\).
The main inequality, outlined in \cref{thm:PG:ineq}, extends the one in \cite[Eq. (2.8)]{latafat2023adaptive} by blending it with \cref{thm:MMtrick}.
This combination is achieved by adding and subtracting a multiple of the residual scaled by a newly added parameter \(\xik\).
The proof is otherwise adapted from that of \cite[Lem. 2.2]{latafat2023adaptive}, and is included in full detail in the \ifarxiv appendix\else preprint version\fi.

\begin{theorem}[main {\algnamefont PG} inequality]\label{thm:PG:ineq}%
	Consider a sequence \(\seq{x^k}\) generated by {\algnamefont PG} iterations \eqref{eq:PG} under \cref{ass:PG:P}, and denote \(\rhok*\coloneqq\frac{\gamk*}{\gamk}\).
	Then, for any \(x\in\dom\varphi\), \(\pik,\xik\geq0\) and \(\oldnu_k>0\), \(k\in\N\),
	\begin{align*}
		\U_{k+1}(x)
	\leq{} &
		\U_k(x)
		-
		\gamk(1+\pik\rhok-\pik*\rhok*^2)P_{k-1}(x)
		-
		\tfrac12\|x^k - x^{k-1}\|^2\biggl\{
			1+\xik-\tfrac{1}{\oldnu_k}
	\\
	&
			-\rhok*^2(\oldnu_{k+1}+\xik*)\Bigl[
				\gamk^2L_k^2
				+
				2
				\gamk\lk
				\bigl(\tfrac{\pik*}{\oldnu_{k+1}+\xik*}-1\bigr)
				-
				\bigl(2\tfrac{\pik*}{\oldnu_{k+1}+\xik*}-1\bigr)
			\Bigr]
		\biggr\},
	\numberthis\label{eq:PG:ineq}
	\end{align*}
	where \(\U_k(x)\) is as in \eqref{eq:PG:Uk}.
	In particular, with \(\oldnu_k\equiv1\), if \(\varphi(x)\leq\inf_{k\in\N}\varphi(x^k)\) (for instance, if \(x\in\argmin\varphi\)),
	and
	\begin{equation}\label{eq:PG:rhokUB}
		\ifarticleclass\else\textstyle\fi
		0
	<
		\rhok*^2
	\leq
		\min\set{
			\frac{1+\pik\rhok}{\pik*},
			~
			\frac{
				\xik
			}{
				(1+\xik*)
				\left[
					\gamk^2L_k^2
					+
					2\gamk\lk\left(\frac{\pik*}{1+\xik*}-1\right)
					+
					\left(1-2\frac{\pik*}{1+\xik*}\right)
				\right]_+
			}
		}
	\end{equation}
	(with \(\xik>0\)) holds for every \(k\), then the coefficients of \(P_{k-1}(x)\) and \(\|x^k-x^{k-1}\|^2\) in \eqref{eq:PG:ineq} are negative, \(\U_{k+1}(x)\leq\U_k(x)\) and thus \(\seq{\U_k(x)}\) converges and \(\seq{x^k}\) is bounded.
\end{theorem}
\grayout{%
\ifarxiv
	\begin{proof}
		See \cref{proofthm:PG:ineq}.
	\end{proof}
\fi
}%

Consistently with what was first observed in \cite{malitsky2020adaptive}, inequality \eqref{eq:PG:rhokUB} confirms that stepsizes should both not grow too fast and be controlled by the local curvature of \(f\).
We next show that, under a technical condition on \(\pik\), all that remains to do is ensuring that the stepsizes do not vanish, which is precisely the reason behind the restrictions on the parameters \(\pik\) and \(\xik\) prescribed in \cref{alg:PG:pinuxik}, as \cref{thm:PG:pinuxik} will ultimately demonstrate.
The technical condition turns out to be a controlled growth of \(\pik\), needed to guarantee that a sequence \(\changerho\rhok*\approx\sqrt{\frac{1+\pik\rhok}{\pik*}}\) will eventually stay above 1.

\begin{theorem}[convergence of {\algnamefont PG} with nonvanishing stepsizes]\label{thm:PG:cvg}%
	Consider the iterates generated by \eqref{eq:PG} under \cref{ass:PG:P}, with \(\gamk*=\gamk\rhok*\) complying with \eqref{eq:PG:rhokUB}.
	If \(\pik*\leq1+\pik\) holds for every \(k\) and \(\inf_{k\in\N}\gamk>0\), then:
	\begin{enumerate}
	\item \label{thm:PG:!x*}%
		The (bounded) sequence \(\seq{x^k}\) has exactly one optimal accumulation point.

	\item \label{thm:PG:x*}%
		If, in addition, \(\seq{\pik}\) and \(\seq{\xik}\) are chosen bounded and bounded away from zero, then the entire sequence \(\seq{x^k}\) converges to a solution \(x^\star\in\argmin\varphi\), and \(\U_k(x^\star)\searrow0\).
	\end{enumerate}
\end{theorem}
\grayout{%
\ifarxiv
	\begin{proof}
		See \cref{proofthm:PG:cvg}.
	\end{proof}
\fi
}%

We now turn to the last piece of the puzzle, namely enforcing a strictly positive lower bound on the stepsizes. The following elementary lemma provides the key insight to achieve this.

\begin{lemma}\label{thm:PG:gamk*:lb1}%
	Let \(f\) be convex and differentiable, and consider the iterates generated by \cref{alg:PG:pinuxik}.
	Then, for every \(k\in\N\) such that \(\gamk\lk<1\) it holds that
	\begin{equation}\label{eq:gamgeq:lk}
		\gamk*
	\geq
		\min\set{
			\gamk
			\sqrt{
				\tfrac{1}{\pimax}+\rhok
			},
			~
			\sqrt{
				\tfrac{\ximin\rmin}{\pimax}
			}
			\tfrac{1}{L_k}
		}.
	\end{equation}
\end{lemma}
\begin{proof}
	We start by observing that the assumptions on \(f\) guarantee that \(0\leq\lk\leq L_k\).
	The (squared) second term in the minimum of \cref{state:pinuxik:gamk*} can be lower bounded as follows
	\begin{align*}
		\tfrac{\rk*}{\pik*}
		\tfrac{
			\xik
		}{
			\left[
				\gamk^2L_k^2
				+
				2\gamk\lk\left(\rk*-1\right)
				+
				\left(1-2\rk*\right)
			\right]_+
		}
	\geq{} &
		\tfrac{\ximin}{\pimax}
		\tfrac{
			\rmin
		}{
			\left[
				\gamk^2L_k^2
				+
				2\gamk\lk\left(\rmin-1\right)
				+
				\left(1-2\rmin\right)
			\right]_+
		}
	\\
	={} &
		\tfrac{\ximin}{\pimax}
		\tfrac{
			\rmin
		}{
			\left[
				\gamk^2L_k^2
				-
				\gamk\lk
				+
				(\gamk\lk - 1)\left(2\rmin - 1\right)
			\right]_+
		}
	\geq{}
		\tfrac{\ximin\rmin}{\pimax\gamk^2L_k^2},
		\numberthis \label{eq:lb:temp}
	\end{align*}
	where the first inequality follows from the fact that the left-hand side is increasing with respect to \(\rk*\), and the second inequality follows since \(\rmin \geq \nicefrac12\).
	In turn, the claimed inequality \eqref{eq:gamgeq:lk} follows from the fact that \(\pik*\leq\pik\leq\pimax\) whenever \(\gamk\lk<1\), see \cref{state:pinuxik}.
\end{proof}

This lemma already hints at a potential lower bound for the stepsize, since boundedness of the sequence \(\seq{x^k}\) ensures lower boundedness of the second term in the minimum.
As for the first term, as long as \(\pik\) is upper bounded, the stepsize can only decrease for a controlled number of iterations.
This arguments will be formally completed in the proof of \cref{thm:PG:pinuxik}, where the following notation will be instrumental.

{\changerho
	\begin{definition}\label{def:meps}%
		Let \(\varepsilon>0\).
		With \(\rho_1=\sqrt{\varepsilon}\) and \(\rhot*=\sqrt{\varepsilon+\rhot}\,\) for \(t\geq1\), we denote
		\[
			t_\varepsilon
		\coloneqq
			\max\set{t\in\N}[\rho_1,\dots,\rhot<1]
		\quad\text{and}\quad
			\m(\varepsilon)
		\coloneqq
			\textstyle\prod_{t=1}^{t_\varepsilon}\rhot.
		\]
	\end{definition}

	Notice that \(\m(\varepsilon)\leq1\) and equality holds iff \(\varepsilon\geq1\) (equivalently, iff \(t_\varepsilon=0\)).
	For \(\varepsilon\in(0,1)\), \(t_\varepsilon\) is a well-defined strictly positive integer, owing to the monotonic increase of \(\varrho_t\) and its convergence to the positive root of the equation \(\varrho^2-\varrho-\varepsilon=0\).
	In particular, \(\m(\varepsilon)\leq\varrho_1=\sqrt{\varepsilon}\) and identity holds if \(t_\varepsilon=1\), that is, \(\sqrt{\varepsilon+\sqrt{\varepsilon}}\geq1\) (and \(\varepsilon<1\)).
	This leads to a partially explicit expression
	\begin{equation}\label{eq:meps=}
		\begin{cases}[c @{~~\text{and}~~} c @{\quad} l]
			t_\varepsilon = 1
		&
			\m(\varepsilon) = \sqrt{\min\set{1,\varepsilon}}
		&
			\text{if }
			\varepsilon\geq\tfrac{3- \sqrt5}{2}\approx0.382
		\\[3pt]
			1<t_\varepsilon\leq\bigl\lceil\frac{1}{\varepsilon(2-\varepsilon)}\bigr\rceil
		&
			\sqrt{\varepsilon^{t_\varepsilon}} < \m(\varepsilon) < \sqrt{\varepsilon}
		\otherwise.
		\end{cases}
	\end{equation}
	The bound on \(t_\varepsilon\) in the second case is obtained by observing that
	\[
		1
	>
		\rho_{t_\varepsilon}
	=
		\rho_{t_\varepsilon}-\rho_{t_\varepsilon}^2+\varepsilon+\rho_{t_\varepsilon-1}
	=\dots=
		\sum_{t=1}^{t_\varepsilon}(\rhot-\rhot^2)
		+
		(t_\varepsilon-1)\varepsilon
	\geq
		(t_\varepsilon-1)\varepsilon(1-\varepsilon)
		+
		(t_\varepsilon-1)\varepsilon,
	\]
	where we used the fact that \(\varepsilon\leq\rhot\leq1-\varepsilon\) and thus \(\rhot-\rhot^2\geq\varepsilon(1-\varepsilon)\) for \(t=1,\dots,t_\varepsilon-1\).
	We also remark that the lower bound in the simplified setting of \cref{thm:PG:simplified} pertains to the case when \(t_\varepsilon =1\), since \(\varepsilon = \tfrac1{\rp}\) falls under the first case above.
}%

\grayout{%
\begin{lemma}\label{thm:rk}%
	The parameter \(\rk\) selected at \cref{state:pinuxik} satisfies
	\(
		\tfrac12\leq\rk\leq\tfrac{\pimax}{1+\ximin}
	\)
	for any \(k\).
	In particular, \(\tfrac{\ximin}{2\pimax}\leq\tfrac{\ximin\rk}{\pimax}\leq\frac{\ximin}{1+\ximin}<1\) holds for every \(k\).
\end{lemma}
\begin{proof}
	The lower bound \(\rk\geq\frac12\) is enforced at \cref{state:pinuxik}.
	As for the upper bound, the requirements \(\xik\geq\ximin\) and \(\pik\leq\pimax\) yield that
	\(
		\rk
	\defeq
		\tfrac{\pik}{1+\xik}
	\leq
		\tfrac{\pimax}{1+\ximin}
	\).
\end{proof}
}%

\begin{theorem}[convergence of \cref{alg:PG:pinuxik}]\label{thm:PG:pinuxik}%
	Under \cref{ass:PG:P}, the sequence \(\seq{x^k}\) generated by \cref{alg:PG:pinuxik} converges to a solution \(x^\star\in\argmin\varphi\) and \(\seq{\U_k(x^\star)}\searrow0\).
	Moreover, there exists
	\(
		k_0
	\leq
		2\bigl\lceil\log_{1+\frac{1}{\pimax}}\bigl(\tfrac{1}{\gamma_0L_{f,\V}}\bigr)\bigr\rceil_+
	\)
	such that
	\[
		\gamk
	\geq
		\gamma_{\rm min}
	\coloneqq
		\m(\nicefrac{1}{\pimax})
		\sqrt{\tfrac{\ximin\rmin}{\pimax}}
		\tfrac{1}{L_{f,\V}}
	\quad
		\forall k\geq k_0,
	\]
	where \(\rmin\coloneqq\inf_{k\in\N}\rk\geq\frac{1}{2}\), \(\m({}\cdot{})\) is as in \cref{def:meps} (see also \eqref{eq:meps=}), and \(L_{f,\V}\) is a Lipschitz modulus for \(\nabla f\) on a compact convex set \(\V\) that contains \(\seq{x^k}\).
\end{theorem}
\begin{proof}
	The conditions prescribed in \cref{state:pinuxik} entail that the requirements of \cref{thm:PG:x*} are met, so that the proof reduces to showing the claimed lower bound on \(\seq{\gamk}\).
	Boundedness of the sequence \(\seq{x^k}\) established in \cref{thm:PG:ineq} ensures the existence of \(L_{f,\V}>0\) as in the statement.
	In particular, recall that \(\lk\leq L_k\leq L_{f,\V}\) holds for all \(k\in\N\), cf. \eqref{eq:lk<=ck}.
	\Cref{thm:PG:gamk*:lb1} then yields that
	\begin{equation}\label{eq:PG:alphk<1}
		\gamk\lk<1
	\quad\Rightarrow\quad
		\gamk*
	\geq
		\min\set{
			\gamk
			\sqrt{
				\tfrac{1}{\pimax}+\rhok
			},
			~
			\sqrt{
				\tfrac{\ximin\rmin}{\pimax}
			}
			\tfrac{1}{L_{f,\V}}
		}.
	\end{equation}
	We first show that \(\gamma_{k_0}L_{f,\V}\geq\sqrt{\frac{\ximin\rmin}{\pimax}}\) holds for some \(k_0\geq0\) upper bounded as in the statement.
	To this end, suppose that \(\gamk L_{f,\V}<\sqrt{\frac{\ximin\rmin}{\pimax}}\) for \(k=0,1,\dots,K\).
	The bounds \(\xik\geq\ximin\) and \(\pik\leq\pimax\) enforced in \cref{state:pinuxik} imply that
	\(
		\tfrac12\leq \rmin\leq\rk\leq\tfrac{\pimax}{1+\ximin}
	\)
	for any \(k\).
	In particular, \(\tfrac{\ximin\rmin}{\pimax}\leq\tfrac{\ximin\rk}{\pimax}\leq\frac{\ximin}{1+\ximin}<1\) holds for every \(k\).
	Then, \(\gamk \lk\leq \gamk L_{f,\V}<\sqrt{\frac{\ximin\rmin}{\pimax}}<1\), and \eqref{eq:PG:alphk<1} hold true for all such \(k\), leading to \(\gamk* \geq  \gamk \sqrt{\nicefrac{1}{\pimax}+ \rhok}\) for \(k=0,\dots,K-1\).
	Since \(\rho_0\geq1\), it follows that \(\rhok*=\nicefrac{\gamk*}{\gamk}\geq \sqrt{\nicefrac{1}{\pimax}+1}\) for \(k =0,\ldots, K-1\). Thus,
	\[
		1
	>
		\tfrac{a\rmin}{\pimax}
	>
		(\gamma_KL_{f,\V})^2
	\geq
		\bigl(1+\tfrac{1}{\pimax}\bigr)
		(\gamma_{K-1}L_{f,\V})^2
	\geq\dots\geq
		\bigl(1+\tfrac{1}{\pimax}\bigr)^K
		(\gamma_0L_{f,\V})^2,
	\]
	from which the existence of \(k_0\) bounded as in the statement follows.

	Let \(k\geq k_0\) be an index such that \(\gamk L_{f,\V}\geq\sqrt{\tfrac{\ximin\rmin}{\pimax}}\), and suppose that \(\gamma_{k+t}L_{f,\V}<\sqrt{\tfrac{\ximin\rmin}{\pimax}}\) for \(t = 1,\ldots, T\).
	As before, the inequalities in \eqref{eq:PG:alphk<1} hold true for all such iterates, leading to
	\[
		\rho_{k+t}
	\geq
		\sqrt{
			\tfrac{1}{\pimax}
			+
			\rho_{k+t-1}
		},
	~~
		t=1,\dots,T+1,
	\quad\text{and in particular}\quad
		\rho_{k+1}\geq\sqrt{\tfrac{1}{\pimax}}.
	\]
	It then follows from the definition of \(\m(\varepsilon)\) and \(t_\varepsilon\) as in \cref{def:meps} with \(\varepsilon=\frac{1}{\pimax}\) that \(\gamma_{k+t}=\gamma_{k+t-1}\rho_{k+t}\) can only decrease for at most \(t\leq t_\varepsilon\) iterations (that is, \(T\leq t_\varepsilon\)), at the end of which
	\[
		\textstyle
		\gamma_{k+t}
	=
		\bigl(
			\prod_{i=1}^{t}\rho_{k+i}
		\bigr)
		\gamk
	\geq
		\m(\tfrac{1}{\pimax})
		\gamk
	\geq
		\m(\tfrac{1}{\pimax})
		\sqrt{\tfrac{\ximin\rmin}{\pimax}}
		\tfrac{1}{L_{f,\V}}
	\defeq
		\gamma_{\rm min},
	\]
	and then increases linearly up to when it is again larger than \(\sqrt{\tfrac{\ximin\rmin}{\pimax}}\frac{1}{L_{f,\V}}\), proving that \(\gamk\geq\gamma_{\rm min}\) holds for all \(k\geq k_0\).
\end{proof}

 	\section{A class of adaptive alternating minimization algorithms}
Leveraging an interpretation of {\algnamefont AMA} as the dual of the proximal gradient method, an adaptive variant is developed for solving \eqref{eq:CP} under the following assumptions.

\begin{assumption}[requirements for problem \eqref{eq:CP}]\label{ass:AMA}~
		\begin{enumeratass}[%
			label={A\(\text{\oldstylenums{\arabic*}}^{\hspace{-0.5pt}*}\)},
			ref={\ref{ass:AMA}.A\(\text{\oldstylenums{\arabic*}}^{\hspace{-0.5pt}*}\)},
		]
		\item \label{ass:psi1}%
			\(\func{\psi_1}{\R^n}{\Rinf}\) is proper, closed, locally strongly convex, and 1-coercive;
		\item
			\(\func{\psi_2}{\R^m}{\Rinf}\) is proper, convex and closed;
		\item \label{ass:relint}%
			\(A\in\R^{m\times n}\) and there exists \(x\in \relint \dom \psi_1\) such that \(Ax \in \relint \psi_2\).
		\end{enumeratass}
\end{assumption}

Under these requiremets, problem \eqref{eq:CP} admits a unique solution \(x^\star\), and by virtue of \cite[Thm.s 23.8, 23.9, and Cor. 31.2.1]{rockafellar1970convex} also its dual
\[\tag{D}\label{eq:D}
	\minimize_{y\in\R^m}\conj \psi_1(-\trans Ay)+\conj \psi_2(y)
\]
has solutions \(y^\star\) characterized by
\(
	y^\star\in\partial \psi_2(Ax^\star)
\)
and
\(
	-\trans Ay^\star\in\partial \psi_1(x^\star)
\),
and strong duality holds.
In fact, \Cref{ass:psi1} ensures that the conjugate \(\conj{\psi_1}\) is a (real-valued) locally Lipschitz differentiable function \cite[Thm. 4.1]{goebel2008local}.
We note that the weaker notion of local strong monotonicity of \(\partial\psi_1\) relative to its graph would suffice, and that this minor departure from the reference is used for simplicity of exposition.
Problem \eqref{eq:D} can then be addressed with proximal gradient iterations
\(
	y^+
=
	\prox_{\gamma g}(y-\gamma\nabla f(y))
\)
analized in the previous section, with \(f\coloneqq\conj \psi_1(-\trans Ay)\) and \(g\coloneqq\conj \psi_2\).
In terms of primal variables \(x\) and \(z\), these iterations result in the alternating minimization algorithm.
We here reproduce the simple textbook steps.
First, observe that
\[
	x=\nabla\conj \psi_1(-\trans Ay)
\quad\Leftrightarrow\quad
	-\trans Ay\in\partial \psi_1(x)
\quad\Leftrightarrow\quad
	0\in\trans Ay+ \partial \psi_1(x)
	=
	\partial[\innprod{A{}\cdot{}}{y}+ \psi_1](x).
\]
Hence, by strict convexity,
\(
	x
	=
	\argmin\set{\psi_1+\innprod{A{}\cdot{}}{y}}
\).
By the Moreau decomposition,
\[
	\prox_{\gamma g}(y-\gamma\nabla f(y))
=
	\prox_{\gamma\conj \psi_2}(y+\gamma Ax)
=
	y+\gamma Ax-\gamma\prox_{\nicefrac{\psi_2}{\gamma}}(\gamma^{-1}y+Ax).
\]
{\algnamefont AMA} iterations thus generate a sequence \(\seq{y^k}\) given in \eqref{eq:AMA}, where
\[
	\Lagr_\gamma(x,z,y)
\coloneqq
	\psi_1(x)+\psi_2(z)+\innprod{y}{Ax-z}+\tfrac{\gamma}{2}\|Ax-z\|^2
\]
is the \(\gamma\)-augmented Lagrangian associated to \eqref{eq:CP}.

\begin{myalg}[alg:adaAMA]{{\algnamefont AdaAMA}\(^{\rp,r}\)}%
\begin{tabular}[t]{@{}l@{}}
	Fix \(y^{-1}\in\R^m\) and \(\gamma_0=\gamma_{-1}>0\).
	With \(\lk\) and \(L_k\) as in \eqref{eq:AMA:CL},
	starting from
\\[2pt]
	\hfill\(
		\begin{cases}[r @{{}={}} l@{~}l]
			x^{-1}
			&
			\argmin_{x\in \R^n}\set{\psi_1(x)+\innprod{y^{-1}}{Ax}}
		\\
			z^0
			&
			\argmin_{z\in \R^m}\Lagr_{\gamma_0}(x^{-1},z,y^{-1})
		\\
			y^0
			&
			y^{-1}+\gamma_0(Ax^{-1}-z^0),
		\end{cases}
	\)\hfill
\\[2pt]
	iterate for \(k = 0, 1, \dots\)
\end{tabular}
\begin{subequations}\label{eq:AMA}
	\begin{align}%
		x^k
	={} &
		\argmin_{x\in \R^n}\set{\psi_1(x)+\innprod{y^k}{Ax}}
		\qquad\text{\gray({\footnotesize\({}= \argmin_{x\in \R^n}\Lagr_0(x,z^k,y^k)\)})}
	\\
		\gamk*
	={} &
		\textstyle
		\gamk\min\set{
			\sqrt{%
				\frac{1}{\rp}+\frac{\gamk}{\gamma_{k-1}}
				\vphantom{
					\frac{
						\frac{r}{\rp}
					}{
						\bigl[
							(1-2r)
							+
							\gamk^2 L_k^2
							+
							2\gamk\lk(r-1)
						\bigr]_+
					}
				}
			},
		~
			\sqrt{
				\frac{
					1 - \frac{r}{\rp}
				}{
					\left[
						(1-2r)
						+
						\gamk^2 L_k^2
						+
						2\gamk\lk(r-1)
					\right]_+
				}
			}
		}
	\\
		z^{k+1}
	={} &
		\prox_{\nicefrac{\psi_2}{\gamk*}}(\gamk*^{-1}y^k+Ax^k)
		\qquad\text{\gray({\footnotesize\({}= \argmin_{z\in \R^m}\Lagr_{\gamk*}(x^k,z,y^k)\)})}
	\\
		y^{k+1}
	={} &
		y^k+\gamk*(Ax^k-z^{k+1})
	\end{align}
\end{subequations}
\end{myalg}

The chosen iteration indexing reflects the dependency on the stepsize \(\gamk\): \(x^k\) depends on \(y^k\) but \emph{not} on \(\gamk*\), whereas \(z^{k+1}\) does depend on it.
Moreover, this convention is consistent with the relation
\(
	\nabla f(y^k)=-Ax^k
\).
Local Lipschitz estimates of \(\nabla  f\) as in \eqref{eq:lL} are thus expressed as
\begin{equation}\label{eq:AMA:CL}
	\lk
=
	-\frac{
		\innprod{Ax^k-Ax^{k-1}}{y^k-y^{k-1}}
	}{
		\|y^k-y^{k-1}\|^2
	}
\quad\text{and}\quad
	L_k
=
	\frac{
		\|Ax^k-Ax^{k-1}\|^2
	}{
		\|y^k-y^{k-1}\|^2
	}.
\end{equation}
Being dually equivalent algorithms, convergence of \refadaAMA\ is deduced from that of \refadaPG.

\begin{theorem}
	Under \cref{ass:AMA}, for any \(\rp>r\geq\frac12\) the sequence \(\seq{x^k}\) generated by \refadaAMA\ converges to the (unique) primal solution of \eqref{eq:CP}, and \(\seq{y^k}\) to a solution of the dual problem \eqref{eq:D}.
\end{theorem}

	\section{Numerical simulations}\label{sec:simulations}
Performance of \refadaPG{} with five different parameter choices from \cref{tab:nu1} is reported through a series of experiments on
(i) logistic regression,
(ii) cubic regularization for logistic loss,
(iii) regularized least squares.
The two former simulations use three standard datasets from the LIBSVM library \cite{chang2011libsvm}, while for Lasso synthetic data is generated based on \cite[\S6]{nesterov2013gradient}; for further details the reader is referred to \cite[\S4.1]{latafat2023adaptive} where the same problem setup is used.
When applicable, the following algorithms are included in the comparisons.\footnote{%
	\url{https://github.com/pylat/adaptive-proximal-algorithms-extended-experiments}%
}
\begin{center}
	\begin{tabular}{|>{\algnamefont}l|l|}
	\hline
		PG-ls$^b$    & Proximal gradient method with nonmonotone backtracking

	\\\hline
		Nesterov  & Nesterov's acceleration with constant stepsize $\nicefrac{1}{L_f}$ \cite[\S10.7]{beck2017first}


	\\\hline
		\adaPG    & \cite[Alg. 2.1]{latafat2023adaptive}

	\\\hline
		\adaPG-MM & Proximal extension of \cite[Alg. 1]{malitsky2020adaptive}

	\\\hline
	\end{tabular}
\end{center}

The backtracking procedure in {\algnamefont PG-ls$^b$} is meant in the sense of \cite[\S10.4.2]{beck2017first}, (see also \cite[LS1]{salzo2017variable} and \cite[Alg. 3]{demarchi2022proximal} for the locally Lipschitz smooth case), without enforcing monotonic decrease on the stepsize sequence.
To improve performance, the initial guess for \(\gamk*\) is warm-started as \(b\gamk\), where \(\gamk\) is the accepted value in the previous iteration and \(b\geq1\) is a backtracking factor.
For each simulation we tested all values of \(b\in\set{1,\, 1.1,\,  1.3,\,  1.5,\,  2}\) and only reported the best outcome.

\begin{figure}[t]
	\begin{center}
		\includetikz[width=\linewidth]{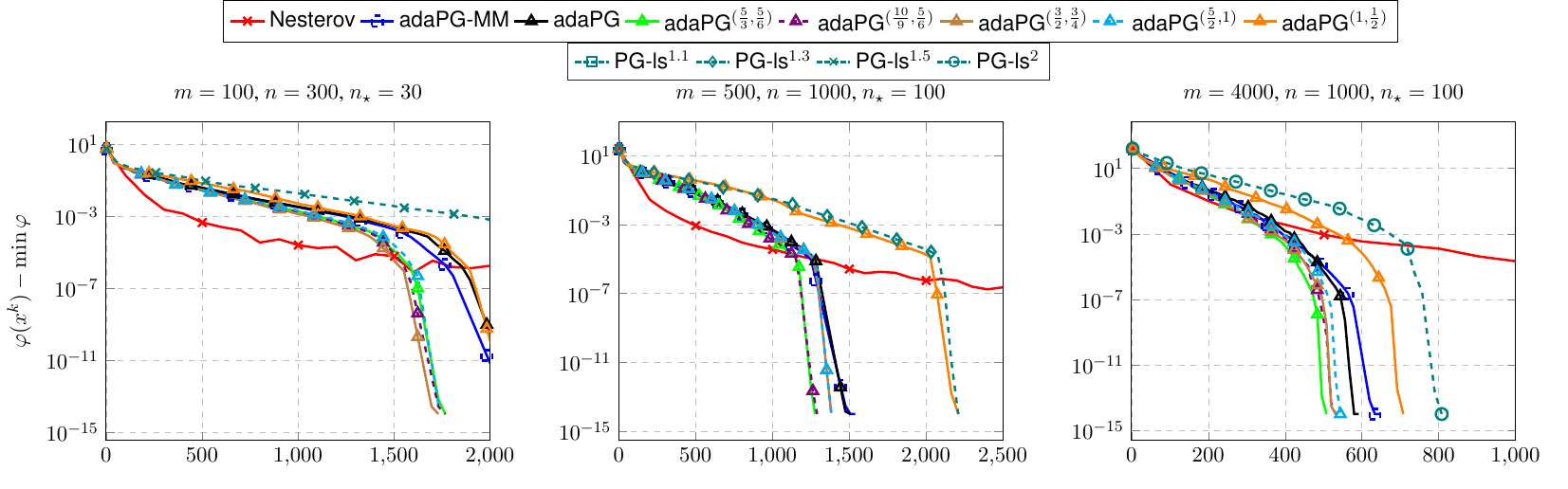}
		\includetikz[width=\linewidth]{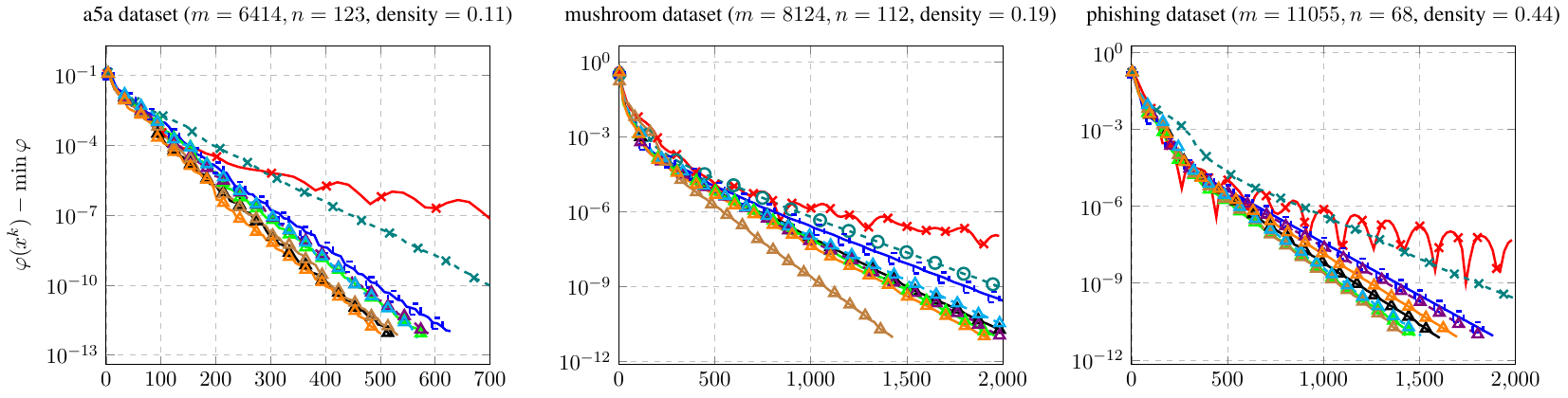}
		\includetikz[width=\linewidth]{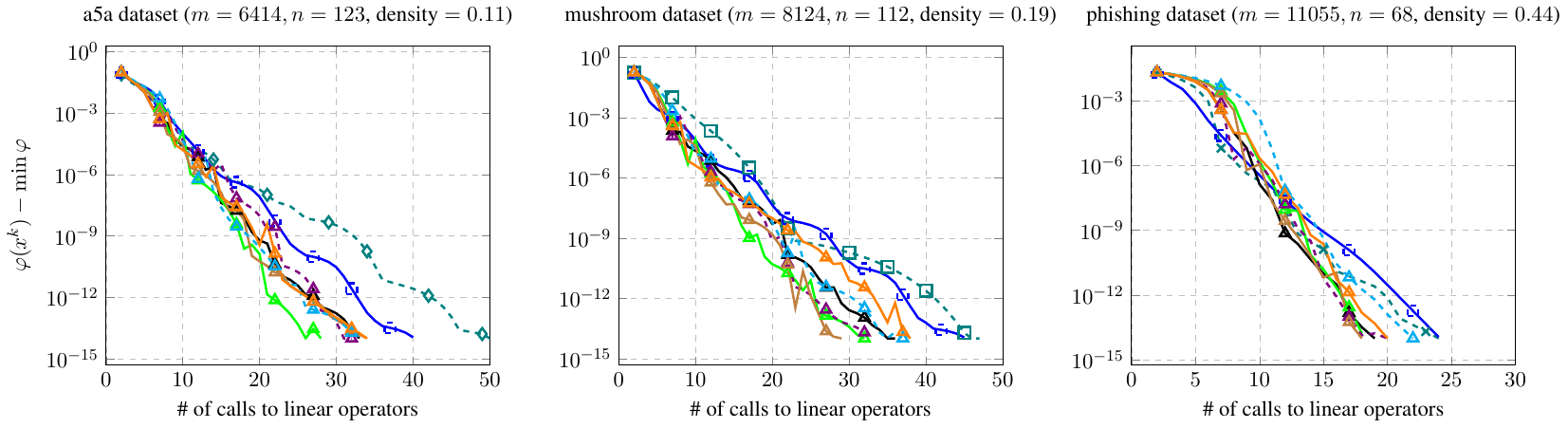}
	\end{center}
	\caption{%
		First row: regularized least squares,
		second row: \(\ell_1\)-regularized logistic regression,
		third row: cubic regularization with Hessian generated for the logistic loss problem evaluated at zero.
		For the linesearch method {\protect\algnamefont PG-ls\(^b\)}, in each simulation only the best outcome for $b\in\protect\set{1,\, 1.1,\,  1.3,\,  1.5,\,  2}$ is reported.
	}%
	\label{fig:PG}%
\end{figure}

	\section{Conclusions}
This paper proposed a general framework for a class of adaptive proximal gradient methods, demonstrating its capacity to extend and tighten existing results when restricting to certain parameter choices.
Moreover, application of the developed method was explored in the dual setting which led to a class of novel adaptive alternating minimization algorithms.

Future research directions include extensions to nonconvex problems, variational inequalities, and simple bilevel optimization expanding upon \cite{malitsky2020golden} and \cite{latafat2023adabim}.
It would also be interesting to investigate the effectiveness of time-varying parameters in our framework for further improving performance and worst-case convergence rate guarantees.

	\ifarticleclass
		\clearpage
		\bibliographystyle{plain}
	\else
		\acks{\TheFunding}%
		\vfill
		\clearpage
		\phantomsection
		\addcontentsline{toc}{section}{References}%
	\fi

	\bibliography{TeX/References.bib}

	\ifarxiv
		\clearpage
		\begin{appendix}
			\crefalias{section}{appendix}%
			\phantomsection
			\addcontentsline{toc}{section}{Appendix}%

			\proofsection{thm:MMtrick}
				We have
{\ifarticleclass\else\mathtight[0.8]\fi
\begin{align*}
	\innprod{\tfrac{\Hk(x^{k-1})-\Hk(x^k)}{\gamk}}{x^k-x^{k+1}}
={} &
	\innprod{\tfrac{\Hk(x^{k-1})-x^k}{\gamk}+ \nabla f(x^k)}{x^k-x^{k+1}}
\\
={} &
	\tfrac{1}{\gamk*}
	\|x^{k+1} - x^k\|^2
	+
	\innprod{
		x^k - x^{k+1}
	}{
		\underbracket*[0.5pt]{
			\tfrac{\Hk(x^{k-1})-x^k}{\gamk\vphantom{\gamk*}}
		}_{\in\partial g(x^k)}
		-
		\underbracket*[0.5pt]{
			\tfrac{\Hk*(x^k)-x^{k+1}}{\gamk*}
		}_{\in\partial g(x^{k+1})}
		}
\\
\geq{}&
	\tfrac{1}{\gamk*}
	\|x^{k+1} - x^k\|^2,
\end{align*}
}%
owing to monotonicity of \(\partial g\).
Invoking the Cauchy-Schwarz inequality completes the proof.

			\proofsection{thm:PG:ineq}
Following \cite{latafat2023adaptive}, we denote
\(
	\Hk\coloneqq\id-\gamk\nabla f
\),
and remind that the sequence generated by \eqref{eq:PG} is characterized by the inclusion
\begin{equation}
\label{eq:PG:subgrad}
	v^k
\coloneqq
	\tfrac{\Hk(x^{k-1})-x^k}{\gamk}
=
	\tfrac{x^{k-1}-x^k}{\gamk}-\nabla f(x^{k-1})
\in
	\partial g(x^k).
\end{equation}
Then,
{\ifarticleclass\else\mathtight[0.88]\fi
	\begin{align*}
		0
	\leq{} &
		g(x)-g(x^{k+1})
		+
		\innprod{\nabla f(x^k)}{x-x^{k+1}}
		-
		\tfrac{1}{\gamk*}\innprod{x^k-x^{k+1}}{x-x^{k+1}}
	\\
	={} &
		g(x)-g(x^{k+1})
		+
		\underbracket[0.5pt]{
			\innprod{\nabla f(x^k)}{x-x^{k+1}}
		}_{\text{(A)}}
		+
		\ifarticleclass
			\tfrac{1}{2\gamk*}\|x^k-x\|^2
			-
			\tfrac{1}{2\gamk*}\|x^{k+1}-x\|^2
			-
			\tfrac{1}{2\gamk*}\|x^k-x^{k+1}\|^2
		\else
			\tfrac{1}{2\gamk*}
			\left(
				\|x^k-x\|^2
				-
				\|x^{k+1}-x\|^2
				-
				\|x^k-x^{k+1}\|^2
			\right)
		\fi
	\end{align*}
}%
holds for any \(x\).
We next proceed to upper bound the term (A) as
{\mathtight[0.74]%
	\begin{align*}
		\text{(A)}
	={} &
		\innprod{\nabla f(x^k)}{x-x^k}
		+
		\innprod{\nabla f(x^k)}{x^k-x^{k+1}}
	\\
	={} &
		\innprod{\nabla f(x^k)}{x-x^k}
		+
		\tfrac{1}{\gamk}\innprod{\Hk(x^{k-1})-x^k}{x^{k+1}-x^k}
		+
		\tfrac{1}{\gamk}\innprod{\Hk(x^{k-1})-\Hk(x^k)}{x^k-x^{k+1}}
	\\
	\overrel[\leq]{\eqref{eq:PG:subgrad}}{} &
	\ifarticleclass
		\fillwidthof[c]{\innprod{\nabla f(x^k)}{x-x^k}}{
			f(x) - f(x^k)
		}
		+
		\fillwidthof[c]{\tfrac{1}{\gamk}\innprod{\Hk(x^{k-1})-x^k}{x^{k+1}-x^k}}{
			g(x^{k+1})-g(x^k)
		}
		+
		\underbracket[0.5pt]{
			\tfrac{1}{\gamk}\innprod{\Hk(x^{k-1})-\Hk(x^k)}{x^k-x^{k+1}}
		}_{\text{(B)}}.
	\else
			f(x) - f(x^k)
		+
			g(x^{k+1})-g(x^k)
		+
			\underbracket[0.5pt]{
				\tfrac{1}{\gamk}\innprod{\Hk(x^{k-1})-\Hk(x^k)}{x^k-x^{k+1}}
			}_{\text{(B)}}.
	\fi
	\numberthis\label{eq:PG:ineqA}
	\end{align*}
}%
We bound the term (B) by Young's inequality with parameter \(\varepsilon_{k+1}\) as
\begin{align*}
	\text{(B)}
\leq{} &
	\tfrac{\varepsilon_{k+1}}{2\gamk}\|x^k-x^{k+1}\|^2
	+
	\tfrac{1}{2\varepsilon_{k+1}\gamk}\|\Hk(x^{k-1})-\Hk(x^k)\|^2
\\
={} &
	\left(\tfrac{\varepsilon_{k+1}}{2\gamk} - \tfrac{\xik*}{2\gamk*}\right)\|x^k-x^{k+1}\|^2
	+
	\tfrac{\xik*}{2\gamk*}\|x^k-x^{k+1}\|^2
	+
	\tfrac{1}{2\varepsilon_{k+1}\gamk}\|\Hk(x^{k-1})-\Hk(x^k)\|^2
\\
\overrel[\leq]{\tiny\cref{thm:MMtrick}}{} &
	\left(\tfrac{\varepsilon_{k+1}}{2\gamk} - \tfrac{\xik*}{2\gamk*}\right)\|x^k-x^{k+1}\|^2
	+
	\left(\tfrac{1}{2\varepsilon_{k+1}\gamk} + \tfrac{\xik*}{2\gamk*}\rhok*^2 \right)\|\Hk(x^{k-1})-\Hk(x^k)\|^2
\\
={} &
	\left(\tfrac{\varepsilon_{k+1}}{2\gamk} - \tfrac{\xik*}{2\gamk*}\right)\|x^k-x^{k+1}\|^2
	+
	\left(
		\tfrac{1}{2\varepsilon_{k+1}\gamk}
		+
		\tfrac{\xik*}{2\gamk*}\rhok*^2
	\right)
	(1-2\gamk\lk+\gamk^2L_k^2)
	\|x^{k-1}-x^k\|^2,
\end{align*}
where we added and subtracted \(\tfrac{\xik*}{2\gamk*}\|x^k - x^{k+1}\|^2\) in the first identity, and used the relation
\(
	\|\Hk(x^k)-\Hk(x^{k-1})\|^2
=
	(1-2\gamk\lk+\gamk^2L_k^2)
	\|x^k-x^{k-1}\|^2
\),
see \cite[Lem. 2.1(ii)]{latafat2023adaptive}, in the second.
We emphasize once again that the introduction of \(\xik*\) allowed us to apply \cref{thm:MMtrick} on the added term, this being the only departure in this proof from that of \cite[Lem. 2.2]{latafat2023adaptive}.
With \(v^k\in\partial g(x^k)\) as in \eqref{eq:PG:subgrad}, we have that
\begin{align*}
	0
\leq{} &
	\thetk*\left(
		\varphi(x^{k-1})-\varphi(x^k)
		-
		\innprod{\nabla f(x^k)+v^k}{x^{k-1}-x^k}
	\right)
\\
={} &
	\thetk*\left(
		\varphi(x^{k-1})-\varphi(x^k)
		-
		\tfrac{1}{\gamk}\|x^k-x^{k-1}\|^2
		+
		\innprod{\nabla f(x^{k-1})-\nabla f(x^k)}{x^{k-1}-x^k}
	\right)
\\
={} &
	\thetk*\left(
		\varphi(x^{k-1})-\varphi(x^k)
		-
		\tfrac{1-\gamk\lk}{\gamk}\|x^k-x^{k-1}\|^2
	\right).
	\numberthis\label{eq:PG:subgradineq}
\end{align*}
By summing the last two inequalities, multiplying by \(\gamk*\), and observing that
\[
	\varphi(x)-\varphi(x^k)+\thetk*(\varphi(x^{k-1})-\varphi(x^k))
=
	\thetk*P_{k-1}(x)
	-
	(1+\thetk*)P_k(x),
\]
we obtain
\begin{align*}
	&
	\tfrac{1}{2}\|x^{k+1}-x\|^2
	+
	\gamk*(1+\thetk*)P_k(x)
	+
	\tfrac{1 + \xik*-\varepsilon_{k+1}\rhok*}{2}\|x^k-x^{k+1}\|^2
\\
\leq{} &
	\tfrac{1}{2}\|x^k-x\|^2
	+
	\thetk*\gamk*P_{k-1}(x)
\\
&
	+
	\rhok*\left(
		(1-2\gamk\lk+\gamk^2L_k^2)
		\left(
			\tfrac{1}{2\varepsilon_{k+1}}
			+
			\tfrac{\xik*}{2}\rhok*
		\right)
		-
		\thetk*(1-\gamk\lk)
	\right)
	\|x^{k-1}-x^k\|^2.
\end{align*}
Selecting \(\thetk*=\pik*\rhok*\) and \(\varepsilon_{k+1}=\nicefrac{1}{\oldnu_{k+1}\rhok*}\), results in the claimed inequality \eqref{eq:PG:ineq}.

Negativity of the coefficients is precisely what the upper bound on \(\rhok*^2\) enforces, and the monotonic decrease of \(\seq{\U_k(x)}\) then follows.
In turn, boundedness of the sequence follows from the fact that
\(
	\frac12\|x^k-x^\star\|^2
\leq
	\U_k(x^\star)
\leq
	\U_0(x^\star)
\)
holds for every \(k\in\N\) and \(x^\star\in\argmin\varphi\).

			\section{Useful lemmas}
This section contains some results which will be useful in the proof of \cref{thm:PG:cvg} given in \cref{proofthm:PG:cvg}.

{%
	\renewcommand{\rho}{\varrho}%
	\begin{lemma}\label{thm:infrhok>1}%
		Let \(\seq{\pik}\) be a sequence contained in an interval \([\pimin,\pimax]\) with \(\pimin>0\) and such that \(\pik*\leq1+\pik\).
		Starting from some \(\rho_0>0\), let \(\seq{\rhok}\) satisfy the relation \(\rhok*^2\geq\frac{1}{\pik*}(1+\pik\rhok)-\varepsilon_{k+1}\) for every \(k\), where \(0\leq\varepsilon_k\to0\) as \(k \to \infty\).
		Then, \(\liminf_{k\to\infty}\rhok>1\).
	\end{lemma}
	\begin{proof}
		We first show that there exist \(\delta>0\) and \(t \geq 1\) such that \(\rhok\geq1+\delta\) holds at least once every \(t\) iterations.
		Up to discarding early iterates, \(\varepsilon_k\leq\frac{1}{2\pimax}\) holds for all \(k\).
		Then,
		\[
			\rhok^2 \geq \tfrac{1}{\pik*}(1+\pik\rhok)-\varepsilon_{k+1} \geq \tfrac{1}{2\pimax}
		\quad
			\forall k,
		\]
		implying that for \(\delta\) small enough \(\rhok-\delta >0\) for all \(k\).
		Fix such a \(\delta>0\), and let \(\bar k\in\N\) and \(t_{\bar k}\geq1\) be such that \(\rho_{\bar k+1},\dots,\rho_{\bar k+t_{\bar k}}\leq1+\delta\).
		Then,
		\(
			\rho_{\bar k+j}-\delta
		\geq
			(\rho_{\bar k+j}-\delta)^2
		\geq
			\rho_{\bar k+j}^2-2\delta\rho_{\bar k+j}
		\),
		hence \(\rho_{\bar k+j}\geq\frac{1}{1+2\delta}\rho_{\bar k+j}^2\), holds for every \(j=1,\dots,t_{\bar k}\), which gives
		\begin{align*}
			1+\delta
		\geq{} \ifarticleclass\else\myampersand\fi
			\rho_{\bar k+t_{\bar k}}
		\geq{} \ifarticleclass\myampersand\fi
			\tfrac{1}{1+2\delta}
			\left(
				\tfrac{1}{\pi_{\bar k+t_{\bar k}}}
				+
				\tfrac{\pi_{\bar k+t_{\bar k}-1}}{\pi_{\bar k+t_{\bar k}}}
				\rho_{\bar k+t_{\bar k}-1}
				-
				\varepsilon_{\bar k+t_{\bar k}}
			\right)
		\\
		\geq{} &
			\tfrac{1}{1+2\delta}
			\left(
				\tfrac{1}{\pi_{\bar k+t_{\bar k}}}
				+
				\tfrac{\pi_{\bar k+t_{\bar k}-1}}{\pi_{\bar k+t_{\bar k}}}
				\tfrac{1}{1+2\delta}
				\Bigl[
					\tfrac{1}{\pi_{\bar k+t_{\bar k}-1}}
					+
					\tfrac{\pi_{\bar k+t_{\bar k}-2}}{\pi_{\bar k+t_{\bar k}-1}}
					\rho_{\bar k+t_{\bar k}-2}
					-
					\varepsilon_{\bar k+t_{\bar k}-1}
				\Bigr]
				-
				\varepsilon_{\bar k+t_{\bar k}}
			\right)
		\\
		={} &
			\tfrac{1}{\pi_{\bar k+t_{\bar k}}}
			\left(
				\tfrac{1}{1+2\delta}
				+
				\tfrac{1}{(1+2\delta)^2}
			\right)
			+
			\tfrac{1}{(1+2\delta)^2}
			\tfrac{\pi_{\bar k+t_{\bar k}-2}}{\pi_{\bar k+t_{\bar k}}}
			\rho_{\bar k+t_{\bar k}-2}
			-
			\bigl(
				\tfrac{1}{1+2\delta}
				\varepsilon_{\bar k+t_{\bar k}}
				+
				\tfrac{1}{(1+2\delta)^2}
				\tfrac{\pi_{\bar k+t_{\bar k}-1}}{\pi_{\bar k+t_{\bar k}}}
				\varepsilon_{\bar k+t_{\bar k}-1}
			\bigr)
		\\
		\cdots\geq{} &
			\tfrac{1}{\pi_{\bar k+t_{\bar k}}}
			\sum_{j=1}^{t_{\bar k}}
			\tfrac{1}{(1+2\delta)^j}
			+
			\tfrac{1}{(1+2\delta)^{t_{\bar k}}}
			\tfrac{\pi_{\bar k}}{\pi_{\bar k+t_{\bar k}}}
			\rho_{\bar k}
			-
			\sum_{j=1}^{t_{\bar k}}
			\tfrac{1}{(1+2\delta)^j}
			\tfrac{\pi_{\bar k+t_{\bar k}+1-j}}{\pi_{\bar k+t_{\bar k}}}
			\varepsilon_{\bar k+t_{\bar k}+1-j}
		\\
		\geq{} &
			\tfrac{1}{\pimax}
			\sum_{j=1}^{t_{\bar k}}
			\tfrac{1-\pimax\varepsilon_{\bar k+t_{\bar k}+1-j}}{(1+2\delta)^j}
		\geq
			\tfrac{1}{2\pimax}
			\sum_{j=1}^{t_{\bar k}}
			\tfrac{1}{(1+2\delta)^j},
		\end{align*}
		where in the last inequality we remind that \(\varepsilon_j\leq\frac{1}{2\pimax}\) for all \(j\).
		By further restricting \(\delta>0\) if necessary, it is apparent that \(t_{\bar k}\) cannot be arbitrarily large, \ie, \(t_{\bar k}\leq\bar t\) for all \(\bar k\in\N\).
		Let \(K=\set{k_1,k_2,\dots}\) be the (infinite) set of all indices \(k\) satisfying \(\rhok>1+\delta\); then,
		\[
			k_{i+1}-k_i\leq\bar t
			\quad
			\forall i.
		\]
		By further discarding early iterates if necessary, we may assume that \(\varepsilon_k\leq\frac{\pimin}{2\pimax}\delta\) holds for every \(k\).
		Then,
		\begin{align*}
			\rho_{k_i+1}
		\geq
			\sqrt{\tfrac{1}{\pi_{k_i+1}}+\tfrac{\pi_{k_i}}{\pi_{k_i+1}}\rho_{k_i}-\varepsilon_{k_i+1}}
		\geq{} &
			\sqrt{\tfrac{1}{\pi_{k_i+1}}+\tfrac{\pi_{k_i}}{\pi_{k_i+1}}(1+\delta)-\varepsilon_{k_i+1}}
		\\
		={} &
		\textstyle
			\sqrt{\underbracket*[0.5pt]{\tfrac{1+\pi_{k_i}}{\pi_{k_i+1}}}_{\geq1}+\tfrac{\pi_{k_i}}{\pi_{k_i+1}}\delta-\varepsilon_{k_i+1}}
		\geq
			\sqrt{1+\tfrac{\pimin}{2\pimax}\delta}
		=
			1
			+
			\delta',
		\end{align*}
		where
		\(
			\delta'
		\coloneqq
			\sqrt{1+\frac{\pimin}{2\pimax}\delta}-1
		\in
			(0,\delta)
		\).
		By the exact same arguments, up to possibly discarding early iterates such that \(\varepsilon_k\leq\frac{\pimin}{2\pimax}\delta'\) holds for every \(k\), we have that
		\[
		\textstyle
			\rho_{k_i+2}
		\geq
			\sqrt{1+\frac{\pimin}{2\pimax}\delta'}
		\eqqcolon
			1+\delta''.
		\]
		Iterating \(\bar t\) times, we have that
		\[
			\rho_{k_i+j}\geq1+\delta_j
		\quad
			\forall i\in\N,\ j=1,\dots,\bar t,
		\]
		where \(\delta_0=\delta\) and
		\(
			\delta_j
		=
			\sqrt{1+\frac{\pimin}{2\pimax}\delta_{j-1}}-1
		\in
			(0,\delta_j)
		\)
		for \(j=1,\dots,\bar t\).
		In particular,
		\[
			\rho_{k_i+j}
		\geq
			1+\delta_j
		\geq
			1+\delta_{\bar t}
		\quad
			\forall j=1,\dots,k_{i+1}-k_i-1
		\]
		and by the definition of \(k_i\), \(i\in\N\), we conclude that \(\rhok\geq1+\delta_{\bar t}>1\) holds for all \(k\) large enough.
	\end{proof}
}%

\begin{lemma}\label{thm:Pkto0}%
	Let \(\varphi=f+g\) where \(f\) is continuously differentiable, and \(g\) is convex and lsc.
	Suppose that a sequence \(\seq{x^k}\) converges to a point \(x^\star\in\argmin\varphi\), and for every \(k\) let \(\bar x^k\coloneqq\prox_{\gamk*g}(x^k-\gamk*\nabla f(x^k))\) with \(\seq{\gamk}\subset\R_{++}\) bounded and bounded away from zero.
	Then, \(\seq{\bar x^k}\) too converges to \(x^\star\).
	Additionally, if either \(\varphi(x^k)\to\min\varphi\) or \(\inf\gamk>0\), then \(\seq{\varphi(\bar x^k)}\to\min\varphi\).
\end{lemma}
\begin{proof}
	The proof is standard.
	Using nonexpansiveness of \(\prox_{\gamk*g}\) and optimality of \(x^\star\),
	\begin{align*}
		\|\bar x^k-x^\star\|
	={} &
		\|\prox_{\gamk*g}(x^k-\gamk*\nabla f(x^k))-\prox_{\gamk*g}(x^\star-\gamk*\nabla f(x^\star))\|
	\\
	\leq{} &
		\|x^k-x^\star\|+\gamk*\|\nabla f(x^k)-\nabla f(x^\star)\|
		\to
		0
		\quad\text{as }k\to\infty.
	\end{align*}
	Since \(f\) is continuous and \(\bar x^k\to x^\star\), it remains to show that \(g(\bar x^k)\to g(x^\star)\).
	For every \(k\in\N\), convexity of \(g\) and the inclusion \(\tfrac{x^k-\bar x^k}{\gamk*}-\nabla f(x^k)\in\partial g(\bar x^k)\) imply
	\begin{align*}
		g(x^k)
	\geq{} &
		g(\bar x^k)
		+
		\innprod{\tfrac{x^k-\bar x^k}{\gamk*}-\nabla f(x^k)}{x^k-\bar x^k}
	\shortintertext{and}
		g(x^\star)
	\geq{} &
		g(\bar x^k)
		+
		\innprod{\tfrac{x^k-\bar x^k}{\gamk*}-\nabla f(x^k)}{x^\star-\bar x^k}.
	\end{align*}
	If \(\varphi(x^k)\to\min\varphi\), that is, \(g(x^k)\to g(x^\star)\), the first inequality combined with lower semicontinuity of \(g\) concludes the proof.
	If instead \(\seq{\gamk}\) is bounded away from zero, the same conclusion follows from the second inequality.
\end{proof}

			\proofsection{thm:PG:cvg}
The sequence \(\seq{x^k}\) is bounded and the coefficients of \(P_{k-1}(x)\) and \(\|x^k-x^{k-1}\|^2\) in \eqref{eq:PG:ineq} are negative owing to \cref{thm:PG:ineq}.
Let us consider \(x\in\argmin\varphi\), and for notational conciseness let us write \(P_k=P_k(x)=\varphi(x^k)-\min\varphi\).
In particular, telescoping \eqref{eq:PG:ineq} yields that
\begin{equation}\label{eq:PG:gPto0}
	\lim_{k\to\infty}\gamk(1+\pik\rhok-\pik*\rhok*^2)P_{k-1}=0.
\end{equation}
\vspace*{-\baselineskip}%
\begin{proofitemize}
\item \itemref[e]{thm:PG:!x*}~
	We proceed by intermediate claims.
	\begin{claims}%
	\item
		{\it There exists an optimal limit point \(x_\infty\) of \(\seq{x^k}\).}

		Since \(\seq{x^k}\) is bounded, it suffices to show that \(\liminf_{k\to\infty}P_k=0\).
		If \(\limsup_{k\to\infty}(1+\pik\rhok-\pik*\rhok*^2)>0\), then the assertion directly follows from \eqref{eq:PG:gPto0}.
		If \(\seq{\gamk}\) is unbounded, then the claim follows from the fact that
		\begin{equation}\label{eq:PG:gamP}
			\gamk P_{k-1}\leq\U_k(x)\leq\U_1(x).
		\end{equation}
		We now argue that no other case is possible.
		To this end, to arrive to a contradiction, suppose that \(1+\pik\rhok-\pik*\rhok*^2\to0\) and that \(\seq{\gamk}\) is bounded.
		In particular, since it is also bounded away from zero, \(\seq{\rhok}\) too is both bounded and bounded away from zero, and in particular so is \(\seq{\pik}\).
		Then, \cref{thm:infrhok>1} applies and yields that \(\liminf_{k\to\infty}\rhok>1\), hence that \(\seq{\gamk}\) eventually diverges (exponentially fast), contradicting its assumed boundedness.

	\item \label{claim:thm:adaPG:optunique}%
		{\it \(x_\infty\) is the only limit point that belongs to \(\argmin\varphi\).}

		Suppose that \(x_\infty'\in\argmin\varphi\) is a limit point.
		Since \(\seq{\U_k(x_\infty')}\) and \(\seq{\U_k(x_\infty)}\) are both convergent, then so is
		\[
			\innprod{x^k}{x_\infty - x_\infty'}
		=
			\U_k(x_\infty') - \U_k(x_\infty) + \tfrac12\|x_\infty\|^2 - \tfrac12\|x_\infty'\|^2.
		\]
		Passing to the limit along the two converging subsequences thus yields
		\(
			\innprod{x_\infty}{x_\infty - x_\infty'}
		=
			\innprod{x_\infty'}{x_\infty - x_\infty'}
		\),
		which after rearranging results in \(\|x_\infty - x_\infty'\|^2=0\), establishing the claim.
	\end{claims}

\item \itemref[e]{thm:PG:x*}~
	Let \(x^\star\) be the optimal accumulation point of \(\seq{x^k}\) (exactly one exists as shown in the previous assertion).
	Note that the proof is complete once we show that \(\U_k(x^\star)\to0\).
	Since this limit exists, we remark that if \(\seq{\gamk}\) is bounded (and so are \(\seq{\pik}\) and \(\seq{\xik}\)), then for a set of indices \(K\subseteq\N\) such that \(\seq{x^k}[k\in K]\to x^\star\) one has that \(\seq{\|x^{k+2}-x^\star\|^2}[k\in K]\), \(\seq{P_{k+1}}[k\in K]\), and \(\seq{\|x^{k+2}-x^{k+1}\|^2}[k\in K]\) all vanish by virtue of \cref{thm:Pkto0} (applied twice), and thus so does \(\seq{\U_{k+2}(x^\star)}[k\in K]\).
%
	In what follows, we assume that \(\seq{\pik}\) and \(\seq{\xik}\) are bounded and bounded away from zero, but make no assumption on boundedness of \(\seq{\gamk}\).
	We then denote \(\pimin\coloneqq\inf_k\pik\), \(\pimax\coloneqq\sup_k\pik\), and \(\ximin\coloneqq\inf_k\xik\),  which are all finite and strictly positive quantities by assumption.

	Observe that, since \(\seq{\pik}\) is bounded and bounded away from zero, it follows from the inequality \(\rhok*^2 \leq \frac{1+\pik\rhok}{\pik*}\) that
	\begin{equation}\label{eq:PG:rhomax}
		\rho_{\rm max}
	\coloneqq
		\sup_{k\in\N}\rhok
	<
		\infty.
	\end{equation}

	Contrary to the claim, suppose that
	\[
		\text{(by contradiction)}
	\quad
		U
	\coloneqq
		\lim_{k\to\infty}\U_k(x^\star)
	>
		0.
	\]
	\begin{claims*}
	\item \label{claim*:PG:gamk*}%
		{\em For any \(K\subseteq\N\),
			\(
				\lim_{K\ni k\to\infty} x^k = x^\star
			\)
			holds iff
			\(
				\lim_{K\ni k\to\infty} \gamk* = \infty
			\).
		}%

		The implication ``\(\Leftarrow\)'' follows from \eqref{eq:PG:gamP}, since \(\seq{x^k}\) is bounded and \(x^\star\) is its unique optimal limit point.
		Suppose now that \(\seq{x^k}[k\in K]\to x^\star\).
		To arrive to a contradiction, up to possibly extracting suppose that \(\seq{\gamk*}[k\in K]\to\bar\gamma<\infty\).
		Then, it follows from \cref{thm:Pkto0} that \(\seq{x^{k+1}}[k\in K]\to x^\star\) and \(\seq{P_{k+1}}[k\in K]\to0\).
		Since \(\seq{\gamma_{k+2}}[k\in K]\) is also bounded owing to \eqref{eq:PG:rhomax}, we may iterate and infer that also \(\seq{x^{k+2}}[k\in K]\) converges to \(x^\star\).
		Recalling the definition of \(\U_k\) in \eqref{eq:PG:Uk}, boundedness (away from zero) of \(\seq{\xik}\) then implies that \(U=\lim_{K\ni k\to\infty}\U_{k+2}(x^\star)=0\), a contradiction.

	\item
		{\em Suppose that \(\seq{x^k}[k\in K]\to x^\star\); then also \(\seq{x^{k-1}}[k\in K]\to x_\star\).}

		It follows from the previous claim that \(\lim_{K\ni k\to\infty}\gamk*=\infty\).
		Because of \eqref{eq:PG:rhomax}, one must also have \(\lim_{K\ni k\to\infty}\gamk=\infty\).
		Invoking again the previous claim, by the arbitrarity of the index set \(K\) the assertion follows.

	\item
		{\em Suppose that \(\seq{x^k}[k\in K]\to x^\star\); then \(\seq{\gamma_{k-1}L_{k-1}}[k\in K]\to\infty\).}

		Using the previous claim twice yields that both \(x^{k-1},x^{k-2}\to x^\star\) as \(K\ni k\to\infty\).
		Hence,
		\[
			\lim_{K\ni k\to\infty}\|x^{k-1}-x^{k-2}\|^2=0.
		\]
		Since \(\seq{\xik}\) is bounded (away from zero), from the expression \eqref{eq:PG:Uk} of \(\U_k\) we then have
		\[
			\lim_{K\ni k\to\infty}\gamk(1+\pik\rhok)P_{k-1}
			=
			U,
		\]
		where we remind that \(U\coloneqq\lim_{k\to\infty}\U_k(x^\star)>0\).
		Denoting \(C\coloneqq\rho_{\rm max}(1+\pimax\rho_{\rm max})\), we have
		\begin{align*}
			\gamma_{k-1} P_{k-1}
		={} &
			\gamma_{k-1}(\varphi(x^{k-1})-\varphi(x^\star))
		\\
		\leq{} &
			\innprod{x^{k-2}-x^{k-1}-\gamma_{k-1}(\nabla f(x^{k-2})-\nabla f(x^{k-1}))}{x^\star-x^{k-1}}
		\\
		\leq{} &
			\innprod{x^{k-2}-x^{k-1}}{x^\star-x^{k-1}}
			+
			\gamma_{k-1}\|\nabla f(x^{k-2})-\nabla f(x^{k-1})\|\|x^\star-x^{k-1}\|
		\\
		={} &
			\innprod{x^{k-2}-x^{k-1}}{x^\star-x^{k-1}}
			+
			\gamma_{k-1}L_{k-1}
			\|x^{k-1}-x^{k-2}\|\|x^\star-x^{k-1}\|
		\end{align*}
		for every \(k\), where the last identity uses the definition \eqref{eq:lL}.
		Then,
		\begin{align*}
			0
		<
			U
		={} &
			\lim_{K\ni k\to\infty}
			\gamk(1+\pik\rhok)P_{k-1}
		\ifarticleclass\\\fi
		={} \ifarticleclass\myampersand\fi
			\liminf_{K\ni k\to\infty}
			\rhok(1+\pik\rhok)\gamma_{k-1}P_{k-1}
		\\
		\leq{} &
			\rho_{\rm max}(1+\pimax\rho_{\rm max})
			\liminf_{K\ni k\to\infty}
			\gamma_{k-1}L_{k-1}
			\underbracket*[0.5pt]{
				\|x^{k-1}-x^{k-2}\|\|x^\star-x^{k-1}\|
			}_{\to0},
		\end{align*}
		which yields the claim.

	\item \label{claim*:PG:rhokto0}%
		{\em Suppose that \(\seq{x^k}[k\in K]\to x^\star\); then \(\seq{\rhok}[k\in K]\to 0\).}

		Since \(\lk\leq L_k\) by \eqref{eq:lk<=ck}, it follows from the previous claim that \(\nicefrac{\gamk^2L_k^2}{\gamk\lk}\) diverges as \(K\ni k\to\infty\).
		Recalling that \(\rk\coloneqq\frac{\pik}{1+\xik}\), from the \(\gamk\)-update at \cref{state:pinuxik:gamk*}, we have
		\[
			0
		\leftarrow
			\tfrac{1+\xi_{k-1}}{\gamma_{k-1}^2L_{k-1}^2}
		\geq
			\rhok^2\biggl[
				\overbrace*{
					\vphantom{\tfrac{\pik-(1+\xik)}{\nicefrac{\gamk^2L_k^2}{\gamk\lk}}}
					\tfrac{1+\xik-2\pik}{\gamma_{k-1}^2L_{k-1}^2}
				}^{\to0}
				+
				\overbrace*{
					\vphantom{\tfrac{\pik-(1+\xik)}{\nicefrac{\gamk^2L_k^2}{\gamk\lk}}}
					(1+\xik\bigr)
				}^{\geq 1}
				+
				\overbrace*{
					2\tfrac{\pik-(1+\xik)}{\nicefrac{\gamk^2L_k^2}{\gamk\lk}}
				}^{\to0}
			\biggr]
		\quad
			\text{as }K\ni k\to\infty,
		\]
		where boundedness of \(\seq{\pik}\) and \(\seq{\xik}\) were used.
	\end{claims*}

	We now conclude the whole proof by constructing a specific subsequence that yields the sought contradiction.
	Observe first that \(\seq{\gamk}\) is necessarily unbounded, which owes to claim \ref{claim*:PG:gamk*} and the fact that \(x^\star\) is an optimal limit point.
	Then, starting with \(k_0=1\) the sequence recursively defined as
	\[
		k_{i+1}=\min\set{k\geq k_i}[\gamk\geq 2\gamma_{k_i}],
	\quad
		i\in\N,
	\]
	is well defined and such that \(\gamma_{k_i}\to\infty\) as \(i\to\infty\).
	A trivial induction reveals that
	\[
		\gamma_{k_i}
	>
		\max_{1\leq j<k_i}\gamma_j
	\quad
		\forall i\geq1.
	\]
	Moreover, claim \ref{claim*:PG:gamk*} implies that \(x^{k_i-1}\to x^\star\) as \(i\to\infty\) and, in turn, claim \ref{claim*:PG:rhokto0} yields that \(\rho_{k_i-1}\to 0\) as \(i\to\infty\).
	In particular, for \(i\) large enough \(\rho_{k_i-1}\leq\rho_{\rm max}^{-1}\), whence the contradiction
	\[
		\gamma_{k_i-2}<\gamma_{k_i}
	\leq
		\rho_{\rm max}\gamma_{k_i-1}
	=
		\rho_{\rm max}\rho_{k_i-1}\gamma_{k_i-2}
	\leq
		\gamma_{k_i-2}.
	\]
	The proof is completed.
\end{proofitemize}

		\end{appendix}
	\fi
\end{document}